\documentclass[12pt]{amsart}
\usepackage{graphicx, color}

\newcommand{\is}{\int_{S^n}}

\newcommand{\be}{\begin{equation}}
\newcommand{\ee}{\end{equation}}
\newcommand{\la}{\langle}
\newcommand{\ra}{\rangle}
\newcommand{\tl}{\tilde}

\newtheorem{theorem}{Theorem}[section]

\theoremstyle{definition}

\theoremstyle{remark}
\newtheorem{remark}[theorem]{Remark}

\numberwithin{equation}{section}

\begin{document}
\setlength{\baselineskip}{1.2\baselineskip}

\title[gradient estimates of mean curvature flow]
{gradient estimates and lower bound for the blow-up time of star-shaped mean curvature flow}

\author{Ling Xiao}

\begin{abstract}

In this paper we consider a star-shaped hypersurface flow by mean curvature. Without any assumption on the convexity,
we give a new proof of gradient estimate for a short time. As an application, we also give a lower bound for the blowing up time.

\end{abstract}

\maketitle
\section{introduction}
\label{sec0}
\setcounter{equation}{0}

Let $F_0:M_n\rightarrow\mathbb{R}^{n+1}$ be a smooth immersion of an n-dimensional hypersurface in Euclidean space.
$M_0=F_0(M_n)$ is a star-shaped hypersurface and $F: M_n\times [0, T)\rightarrow\mathbb{R}^{n+1}$ satisfying
\be\label{eq1.1}
\begin{aligned}
\frac{\partial F}{\partial t}(p\,,t)&=-H(p\,,t)\nu(p\,,t)\\
F(\cdot\,, 0)&=M_0.\\
\end{aligned}
\ee
Here $\nu$ is the outer normal vector, $H$ is the mean curvature of $M_t.$

For closed hypersurfaces, it is well known that the solution of \eqref{eq1.1} exists on a finite time interval
$[0, T_c)$, $0<T_c<\infty,$ and the curvature of the hypersurfaces becomes unbounded as $t\rightarrow T_c.$
The study of a detailed description of the singular behavior as $t\rightarrow T_c$ has drawn lots of attentions
for the past decades (see \cite{And 1},  \cite{Ger}, \cite{HK}, \cite{Hui1}, \cite{Hui2}, and \cite{Hui4}). We can see that, the growth rate of the second fundamental form $A$
as $t$ approaches $T_c$ plays an important role in determining the shape of singularities (see \cite{CM1}, \cite{Eck},
\cite{Hui3}, and \cite{HS}).

Inspired by \cite{Tak}, we give a new proof for the short time gradient estimate for $\langle X, \nu\rangle^{-1}$ of $M_t,$
where $M_t$ is a family of smooth closed n-dimensional hypersurfaces immersed in $\mathbb R^{n+1}$ evolving by mean curvature, and $M_0$ is star-shaped.
As an application, we give a lower bound on the blow-up time
under the assumption that $H$ remains bounded through the flow. Our main theories are stated as following:

\begin{theorem}
\label{th1}
Let $M_t,$ $t\in[0, T_c)$ be a solution of equation \eqref{eq1.1}. Then for every $0<T_0<T_c,$
there exists $0<T\leq T_0,$ such that when $t\in[0, T]$ we have
\[\langle X, \nu\rangle^{-1}\leq 2c_3f_0,\]
here $T=\min\{\frac{1}{Gf_0^4}, T_0\},$ $f_0=\max_{M_0}\langle X, \nu\rangle^{-1},$ $G$ depends on $M_0,$
$T_0,$ and $c_3$ is a constant depending on $M_0.$
\end{theorem}

\begin{theorem}
\label{th2}
Let $M_t,$ $t\in[0, T_c)$ be a solution of \eqref{eq1.1}.
If $\sup_{M_n\times[0, T_c)}H\leq C_H<\infty,$ then $T_c>\frac{1}{24c_2c_3^2f_0^4},$
where $c_2$ and $c_3$ depends on $M_0.$
\end{theorem}

\bigskip

\section{gradient estimates}
\label{sec1}
\setcounter{equation}{0}
\begin{proof}[Proof of Theorem \ref{th1}]
We define
\be\label{eq1.2}
\rho=\rho_{(X_0, t_0)}(X, t)=\frac{1}{\left(4\pi(t_0-t)\right)^{n/2}}
\exp\left\{-\frac{|X-X_0|^2}{4(t_0-t)}\right\}.
\ee
Then, for any $u\in C^2$ we have
\be\label{eq1.3}
\begin{aligned}
&\frac{d}{dt}\int_{M_t}u\rho d\mathcal{H}^n\\
&=\int_{M_t}u_t\rho+u\rho_t-H^2u\rho d\mu_t\\
&=\int_{M_t}\left\{u\left(\left(\frac{d}{dt}+\triangle_{M_t}\right)\rho-H^2\rho\right)\right.
\left.+\rho\left(\frac{d}{dt}-\triangle_{M_t}\right)u\right\}d\mu_t\\
&=-\int_{M_t}u\rho\left|\vec{H}+\frac{1}{2\tau}\vec{F}^\perp\right|^2d\mu_t
+\int_{M_t}\rho\left(\frac{d}{dt}-\triangle_{M_t}\right)ud\mu_t,\\
\end{aligned}
\ee
where $\tau=t_0-t.$

Now let $f=\langle X, \nu\rangle^{-1},$ then it is easy to see that f satisfies
\be\label{eq1.4}
f_t-\triangle f=\frac{-2}{f}|\nabla f|^2+2f^2H-|A|^2f.
\ee

Plugging equation \eqref{eq1.4} into \eqref{eq1.3} we obtain
\be\label{eq1.5}
\begin{aligned}
&\frac{d}{dt}\int_{M_t}f\rho d\mu_t\\
&=-\int_{M_t}f\rho\left|\vec{H}+\frac{1}{2\tau}\vec{F}^\perp\right|^2d\mu_t\\
&-\int_{M_t}\frac{2\rho}{f}|\nabla f|^2-2f^2H\rho+|A|^2f\rho d\mu_t\\
&\leq 2\int_{M_t}f^2H\rho d\mu_t.
\end{aligned}
\ee

Denote $f_{\infty}=\sup_{M_n\times[0, T]}f=f(Y, s),$
$\tilde{\rho}=\rho_{(Y,s)}(X, t),$ $|X_0|=diam(M_0),$
$c_1=2\max_{M_n\times[0, T_0]}H,$ where $T\leq T_0<T_c,$ and $T_c$ is the blow-up time.

Since $M_0$ is a compact star-shaped hypersurface, we know that $c_1>0.$
Therefore, we have
\be\label{eq1.6}
\frac{d}{dt}\int_{M_t}f\tilde{\rho}d\mu_t\leq c_1f^2_\infty\int_{M_t}\tilde{\rho}d\mu_t.
\ee

Next, we are going to estimate $\int_{M_t}\tilde{\rho}d\mu_t.$
When $M_t$ is star-shaped, we can represent $M_t$ as a radial graph over $S_n,$
i.e. for $X(t)\in M_t,$ we have \[X(t)=e^{v(z, t)}z=r(z, t)z,\,z\in S_n.\]
It is easy to see that
\[g_{ij}=e^{2v}\left(\delta_{ij}+v_iv_j\right),\,\mbox{and $\det g=e^{2nv}\left(1+|\nabla v|^2\right)$}.\]
Therefore,
\be\label{eq1.7}
\begin{aligned}
\int_{M_t}\tilde{\rho}d\mu_t&=\is\tilde{\rho}e^{nv}\sqrt{1+|\nabla v|^2}d\mu\\
&=\is\tilde{\rho}\sqrt{1+|\nabla v|^2}r^nd\mu.\\
\end{aligned}
\ee
Since $f=\la X, \nu\ra^{-1}=\frac{\sqrt{1+|\nabla v|^2}}{e^v},$ we get
\be\label{eq1.8}
\begin{aligned}
\int_{M_t}\tilde{\rho}d\mu_t&=\is\tl{\rho}fr^{n+1}dz\\
&=\is\frac{1}{(4\pi\tau)^{n/2}}e^{-\frac{|X(t)-Y|^2}{4\tau}}fr^{n+1}dz\\
&\leq f_\infty\is\frac{1}{(4\pi\tau)^{n/2}}e^{-\frac{|rz-r_0z_0|^2}{4\tau}}r^{n+1}dz.\\
\end{aligned}
\ee
We only need to show
\[Q=\is\frac{1}{(4\pi\tau)^{n/2}}e^{-\frac{|rz-r_0z_0|^2}{4\tau}}r^{n+1}dz\]
is bounded.

Case 1: When $Y$ is the origin, we have
\be\label{eq1.9}
\begin{aligned}
Q&=\is\frac{1}{(4\pi\tau)^{n/2}}e^{-\frac{r^2}{4\tau}}r^{n+1}dz\\
&\leq\frac{\omega_n}{\pi^{n/2}}|X_0|\max_{\phi\in\mathbb{R}_+}e^{-\phi}\phi^{n/2}\\
&\leq\tl{c_1}|X_0|.\\
\end{aligned}
\ee
Case 2: When $Y$ is not the origin, we have\\
a). When $r/r_0\geq 2,$ we have
\[|rz-r_0z_0|^2\geq (r-r_0)^2\geq \frac{r^2}{4}.\]
Thus,
\be\label{eq1.10}
\begin{aligned}
Q&\leq\is\frac{1}{(4\pi\tau)^{n/2}}e^{-\frac{r^2}{16\tau}}r^{n+1}dz\\
&\leq\frac{\omega_n}{\pi^{n/2}}|X_0|\max_{\phi\in\mathbb{R}_+}e^{-\frac{\phi}{4}}\phi^{n/2}\\
&\leq\tl{c_2}|X_0|.
\end{aligned}
\ee
b). When $r/r_0<2,$ we divide $S_n$ into two parts as shown in the graph,and denoted by $S_1,$
$S_2$ respectively.
\def\svgwidth{2 in}
\begingroup%
  \makeatletter%
  \providecommand\color[2][]{%
    \errmessage{(Inkscape) Color is used for the text in Inkscape, but the package 'color.sty' is not loaded}%
    \renewcommand\color[2][]{}%
  }%
  \providecommand\transparent[1]{%
    \errmessage{(Inkscape) Transparency is used (non-zero) for the text in Inkscape, but the package 'transparent.sty' is not loaded}%
    \renewcommand\transparent[1]{}%
  }%
  \providecommand\rotatebox[2]{#2}%
  \ifx\svgwidth\undefined%
    \setlength{\unitlength}{1307.27311683bp}%
    \ifx\svgscale\undefined%
      \relax%
    \else%
      \setlength{\unitlength}{\unitlength * \real{\svgscale}}%
    \fi%
  \else%
    \setlength{\unitlength}{\svgwidth}%
  \fi%
  \global\let\svgwidth\undefined%
  \global\let\svgscale\undefined%
  \makeatother%
  \begin{picture}(1,0.65283272)%
    \put(0,0){\includegraphics[width=\unitlength]{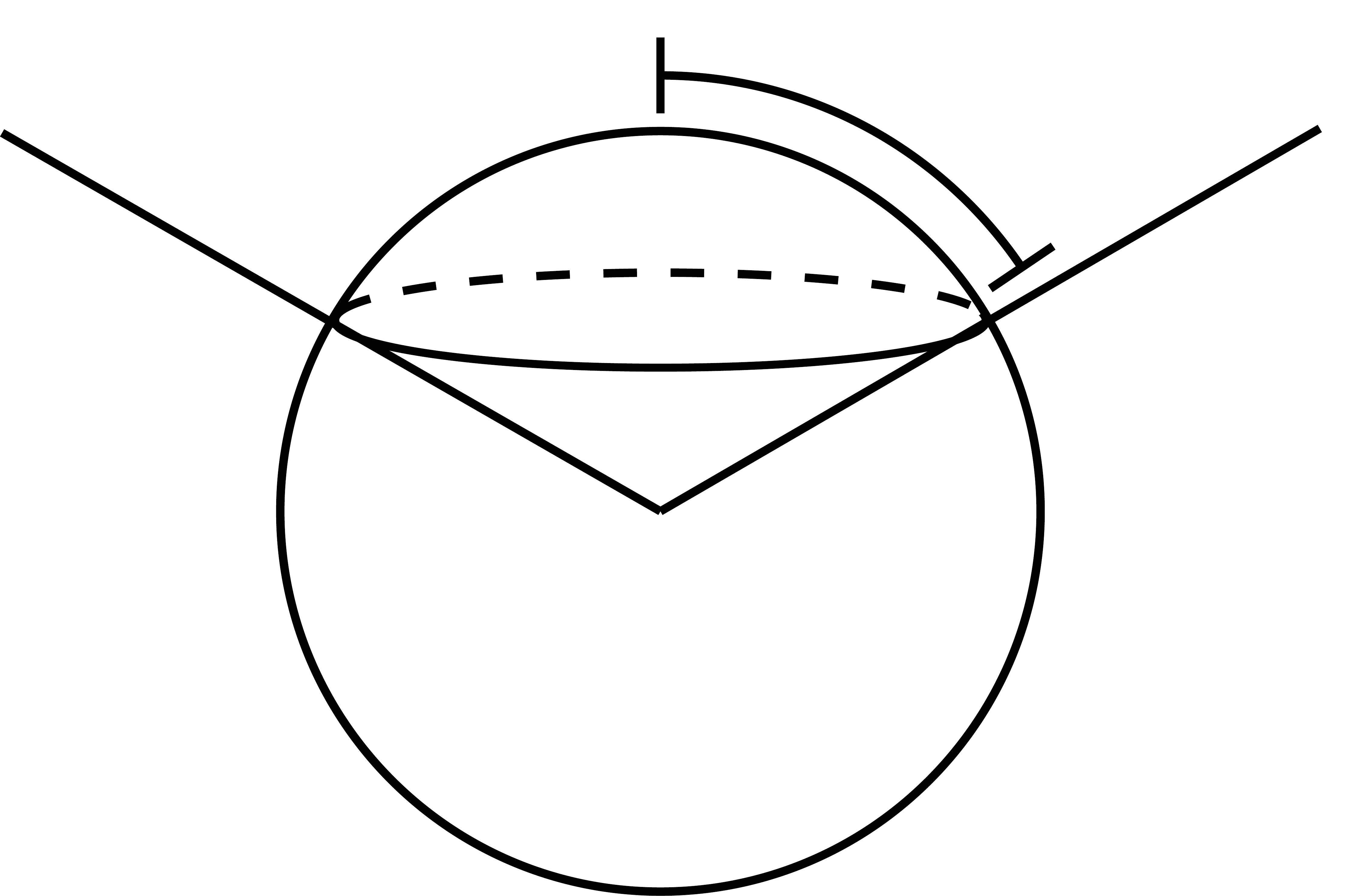}}%
    \put(0.63775436,0.65871778){\color[rgb]{0,0,0}\makebox(0,0)[lt]{\begin{minipage}{1.31404659\unitlength}\raggedright \(\frac{\pi}{3}  \)\end{minipage}}}%
    \put(0.22333474,0.53961436){\color[rgb]{0,0,0}\makebox(0,0)[lt]{\begin{minipage}{1.80873916\unitlength}\raggedright \( S_1 \)\end{minipage}}}%
    \put(0.22333896,0.26373021){\color[rgb]{0,0,0}\makebox(0,0)[lt]{\begin{minipage}{1.48928521\unitlength}\raggedright \(S_2\)\end{minipage}}}%
  \end{picture}%
\endgroup%

Then, we have
\be\label{eq1.11}
\begin{aligned}
Q&=\is\frac{1}{(4\pi\tau)^{n/2}}e^{-\frac{r_0^2|(r/r_0)z-z_0|^2}{4\tau}}r^{n+1}dz\\
&=\int_{S_1}+\int_{S_2}=I+II.\\
\end{aligned}
\ee
In $S_1,$ it's easy to see that
\[\left|(r/r_0)z-z_0\right|^2\geq |z-z_0|^2\cos^2\frac{\pi}{6}=\frac{3}{4}|z-z_0|^2,\]
so we have
\be\label{eq1.12}
I\leq\int_{S_1}\frac{1}{(4\pi\tau)^{n/2}}e^{-\frac{3}{4}r_0^2\frac{|z-z_0|^2}{4\tau}}r^{n+1}dz.
\ee
Since $S_1$ can be represented as a vertical graph over the tangent plane at $z_0,$
we assume $z=\left(x, u(x)\right)\in S_1,$ where $x\in proj(S_1).$
Let $\tl{c_0}=\max_{proj(S_1)}\sqrt{1+|\nabla u|^2},$ then
\be\label{eq1.13}
\begin{aligned}
I&\leq\tl{c_0}\int_{proj(S_1)}\frac{1}{(4\pi\tau)^{n/2}}e^{-\frac{3}{4}r_0^2\frac{|x|^2}{4\tau}}r^{n+1}dx\\
&\leq\tl{c_0}\frac{|X_0|}{\pi^{n/2}}\int_{\mathbb{R}^n}e^{-\frac{3}{4}|x|^2}\left(\frac{r}{r_0}\right)^ndx\\
&\leq\tl{c_3}|X_0|.\\
\end{aligned}
\ee
In $S_2$ we have
\[|(r/r_0)z-z_0|^2\geq 1/2.\]
Therefore,
\be\label{eq1.14}
\begin{aligned}
II&\leq\int_{S_2}\frac{1}{(4\pi\tau)^{n/2}}e^{-\frac{r_0^2}{8\tau}}r^{n+1}dz\\
&\leq2^n|X_0|\frac{|\omega_n|}{\pi^{n/2}}\max_{\phi\in\mathbb{R}_+}e^{(-1/2)\phi}\phi^{n/2}\\
&\leq\tl{c_4}|X_0|.\\
\end{aligned}
\ee
We conclude that
\[\int_{M_t}\tl{\rho}d\mu_t\leq c f_{\infty}|X_0|.\]
Now we have
\be\label{eq1.15}
\frac{d}{dt}\int_{M_t}f\tl{\rho}d\mu_t\leq cc_1|X_0|f^3_\infty=c_2f^3_\infty.
\ee
Let $f_0=\max_{M_0}\la X, \nu\ra^{-1},$ then we have
\[\int_{M_t}f\tl{\rho}d\mu_t|_{t=0}\leq c|X_0|f_0^2\leq c_3f^2_0,\]
$c_3$ is chosen such that $c_3f_0\geq 1.$
Since
\[\lim_{t\nearrow s}\int_{M_t}f\tl{\rho}d_{\mu_t}=f(Y,\,s)=f_\infty,\]
we get
\be\label{eq1.16}
f_\infty-c_3f_0^2\leq\int_0^s\frac{d}{dt}\int_{M_t}f\tl{\rho}d\mu_tdt\leq sc_2f^3_\infty.
\ee
Now consider function $h(r)=sc_2r^3-r+c_3f_0^2,$ we have
\[h(0)=c_3f_0^2>0, \mbox{ and $h(f_0)>0$}.\]
Furthermore,
\[h'(r)=3c_2sr^2-1<-1/2\, \mbox{ for $r\in\left(0, \frac{1}{\sqrt{6c_2s}}\right)$}.\]
We can see if $2c_3f_0^2<\frac{1}{\sqrt{6c_2s}},$ then $h'(r)<-1/2$ for any
$r\in(0, 2c_3f^2_0).$
Therefore, there exists $\alpha\in(f_0, 2c_3f_0^2)$ such that $h(r)<0.$
Let $T=\min\{\frac{1}{24c_2c_3^2f_0^4}, T_0\}$ and $s\leq T,$ we assume $f_\infty>2c_3f_0^2.$
Then, there exists $s'\in(0, s)$ such that
\[v_1:=\max_{M_n\times[0, s']}\la X, \nu\ra^{-1}=\alpha\,\mbox{ and $h(v_1)<0$},\]
leads to a contradiction.
So, we have when $t\leq\min\{\frac{1}{24c_2c_3^2f_0^4}, T_0\},$ $f_{\infty}\leq 2c_3f_0^2.$
\end{proof}

\bigskip

\section{lower bound for blow-up time}
\label{sec2}
\setcounter{equation}{0}
\begin{proof}[Proof of Theorem \ref{th2}]
Now we assume
\[\sup_{M_n\times[0, T_c)}H\leq c_H<\infty,\]
then $c_2=c_Hc|X_0|$ is always bounded.
Let $T=\frac{1}{24c_2c_3^2f_0^4}$ and $T_0\leq T,$ consider
\[\Phi=\log|A|^2+2\log f-Bt\,\mbox{ in $M_n\times[0, T_0]$},\]
where $B>0$ will be determined later.
Denote $g=|A|^2,$ then we have
\[\Phi_t-\triangle\Phi=\frac{1}{g}(g_t-\triangle g)+\frac{2}{f}(f_t-\triangle f)
+\frac{|\nabla g|^2}{g^2}+2\frac{|\nabla f|^2}{f^2}-B.\]
Since
\[\frac{\partial}{\partial t}|A|^2-\triangle|A|^2=-2|\nabla A|^2+2|A|^4,\]
we have
\be\label{eq2.1}
\begin{aligned}
\Phi_t-\triangle\Phi&=\frac{1}{|A|^2}(-2|\nabla A|^2+2|A|^4)\\
&+\frac{2}{f}\left(\frac{-2}{f}|\nabla f|^2+2f^2H-|A|^2f\right)\\
&+\frac{|\nabla g|^2}{g^2}+2\frac{|\nabla f|^2}{f^2}-B\\
&=-2\frac{|\nabla A|^2}{|A|^2}+4fH+\frac{|\nabla g|^2}{g^2}-2\frac{|\nabla f|^2}{f^2}-B.\\
\end{aligned}
\ee
Choose $B>4f_\infty c_H,$
\be\label{eq2.2}
\Phi_t-\triangle\Phi<-2\frac{|\nabla A|^2}{|A|^2}+\frac{|\nabla g|^2}{g^2}-2\frac{|\nabla f|^2}{f^2}.
\ee
At an interior point where $\Phi$ achieves its local maximum, we have
\[\frac{\nabla g}{g}=-2\frac{\nabla f}{f}.\]
Moreover,
\[\frac{\left|\nabla|A|^2\right|^2}{2|A|^4}\leq2\frac{|\nabla A|^2}{|A|^2}.\]
So, we have $\Phi_t-\triangle\Phi<0$ at the point, leads to a contradiction.
We conclude that
\[\Phi\leq C(M_0)\]
and \[T_c>\frac{1}{24c_2c_3^2f_0^4}=\frac{1}{24c_Hc^3|X_0|^3f_0^4}.\]
\end{proof}

\begin{remark}\label{rmk1}
From the proof of Theorem \ref{th2}, we can also conclude that during a mean curvature flow,
$|A|^2$ is staying bounded as long as $\langle X, \nu \rangle^{-1}$ is bounded.
\end{remark}

\end{document}